\documentclass[]{interact}

\usepackage{epstopdf}% To incorporate .eps illustrations using PDFLaTeX, etc.
\usepackage[caption=false]{subfig}% Support for small, `sub' figures and tables

\usepackage[numbers,sort&compress]{natbib}% Citation support using natbib.sty
\bibpunct[, ]{[}{]}{,}{n}{,}{,}% Citation support using natbib.sty
\makeatletter% @ becomes a letter
\def\NAT@def@citea{\def\@citea{\NAT@separator}}% Suppress spaces between citations using natbib.sty
\makeatother% @ becomes a symbol again

\theoremstyle{plain}% Theorem-like structures provided by amsthm.sty
\newtheorem{theorem}{Theorem}[section]
\newtheorem{lemma}[theorem]{Lemma}
\newtheorem{corollary}[theorem]{Corollary}
\newtheorem{proposition}[theorem]{Proposition}

\theoremstyle{definition}
\newtheorem{definition}[theorem]{Definition}
\newtheorem{example}[theorem]{Example}

\theoremstyle{remark}
\newtheorem{remark}{Remark}

\begin{document}

%\articletype{ARTICLE TEMPLATE}% Specify the article type or omit as appropriate

\title{On the Maximal Monotone Operators in Hadamard Spaces}

\author{
\name{Ali Moslemipour\textsuperscript{a}, Mehdi Roohi\textsuperscript{b}\thanks{CONTACT Mehdi Roohi. Email: m.roohi@gu.ac.ir} and Jen-Chih Yao\textsuperscript{c}}
\affil{\textsuperscript{a} Department of Mathematics, Aliabad Katoul Branch, Islamic Azad University, Aliabad Katoul, Iran}
\affil{\textsuperscript{b} Department of Mathematics, Faculty of Sciences,  Golestan University,
Gorgan, Golestan, Iran}
\affil{\textsuperscript{c} Research Center for Interneural Computing, China Medical University Hospital,
China Medical University, Taichung, Taiwan}
}
\maketitle

\begin{abstract}
In this paper, some topics of monotone operator theory in the setting
 of Hadamard spaces are investigated. For a fixed element $p$ in a Hadamard space $X$, the notion of $p$-Fenchel conjugate is introduced and a  type of the Fenchel-Young inequality is proved. Moreover, we examine the $p$-Fitzpatrick transform and its main properties for monotone set-valued operators in Hadamard spaces. Furthermore, some relations between maximal monotone operators and certain classes of proper, convex,  l.s.c. extended real-valued functions on $X\times X^{\scalebox{0.7}{$^{\lozenge}$}}$, are given.
\end{abstract}


\section{Introduction}

Recently, the concept of Hadamard spaces has gained considerable attention. They have been applied to several areas of mathematics, such as convex optimization, fixed point theory, ergodic theory, geometric group theory, computational biology, and also bio-informatics. There are many generalizations of Hilbert spaces with different points of views. A nonlinear generalization of Hilbert spaces are Hadamard spaces.
\newcommand{\cat}{\mathrm{CAT}}
\newcommand*{\smalllozenge}{{\scalebox{0.7}{$^{\lozenge}$}}}
\newcommand*{\loz}{{\scalebox{0.45}{${\lozenge}$}}}

More precisely, Hadamard spaces are complete metric spaces with a non-positive curvature, also known as complete $\cat(0)$ spaces.  Classical real hyperbolic $n$-spaces, $\mathbb{R}$-trees, Hadamard manifolds, Euclidean buildings, nonlinear Lebesgue spaces, Hilbert balls, and complete simply connected Riemannian manifolds of a non-positive sectional curvature are important examples of Hadamard spaces. For more details and other examples, see {\rm{\rm\cite[Chapter II.1, 1.15]{BridsonHaefliger}}}.

It is well known that, the notion of non-positive curvature spaces were mentioned by
 J. Hadamard and E. Cartan in the 1920's.
In 1950, H. Busemann and A.D. Aleksandrov generalized  the concept of geodesic metric spaces based on the concept of manifolds with a non-positive sectional curvature.  Gromov  suggested the notation $\cat(0)$ for a non-positive curvature geodesic metric space. The letters C, A and T in $\cat(0)$ stand for  Cartan, Aleksandrov and Toponogov, respectively.

 Monotone operator theory plays an important role in several different areas of pure and applied mathematics such as nonlinear analysis, nonlinear functional analysis, convex analysis and
convex optimization theory, for more details see \cite{Bauschke2017,Borwein} and the references cited therein. The first application of the concept of
monotone operator was made implicitly by M. Golomb
\cite{Golomb1935},
but the notion of monotone operators was introduced  independently
by Kacurovskii \cite{Kacurovski1960}, Zarantonello \cite{Zarantonello1960} and Minty \cite{Minty1962}. During these years, the theory of monotone operators has found many generalizations and applications, for instance, see \cite{AYMY2022,AAJRR2022,[m],AlizadehRoohi,AliHadjRoohi2012,[71],[89],[1],MPP2015,NaLucThe,[rr],[139],Verona1,Y1994}.

Let $E$ be a Banach space and let $E^*$ be its linear dual space. Representing the notion of maximal monotone set-valued operators from $E$ to $E^*$, in terms of subdifferentials of the skew-symmetric saddle functions on $E\times E$ were introduced and investigated by Krauss \cite{Krauss85a,Krauss85b,Krauss86}. Then Fitzpatrick  \cite{Fitzpatrick} suggested a new way to represent maximal monotone operators by subdifferentials of a certain class of convex functions defined on $E\times E^*$. Martinez-Legaz and Th\'{e}ra \cite{Martinez-LegazThera} and
Burachik and  Svaiter \cite{BurachikSvaiter}, independently, rediscovered Fitzpatrick transform. Some of recent advances on monotone operator theory and the Fitzpatrick transform can be found in
\cite{Borwein,EliasMartinez-Legaz,Martinez-LegazSvaiter}.

The concept of monotone operators from a Hadamard space $X$ to its dual, $X^*$, was introduced and considered in \cite{KakavandiAmini,KhatibzadehRanjbar2017}.
Recently, the notions of monotone relations and maximal monotone relations from a Hadamard space to its linear dual space were introduced and their basic properties were investigated in \cite{MRMA1,MRMA2}.
The  Fitzpatrick transform of monotone relations in Hadamard spaces was introduced in \cite{MRMA2}. There it was shown that the Fitzpatrick transform of a certain class of monotone relations is proper, convex and lower semi-continuous.

Motivated by the above statements and papers, some relations between maximal monotone operators and certain classes of proper, convex,  l.s.c. extended real-valued functions on $X\times X^{\scalebox{0.7}{$^{\lozenge}$}}$, are given, where $X$ is a Hadamard space and $X^{\scalebox{0.7}{$^{\lozenge}$}}$ is its linear dual space.

The next sections of this manuscript are organized as follows. In section \ref{pre}, some basic definitions and results are collected, which are used directly or indirectly in the following sections. More precisely, the notions of geodesic spaces, CN-inequality, $\cat(0)$-spaces, Hadamard spaces, quasilinearization in abstract metric spaces, dual space of a Hadamard space, linear dual space of a Hadamard space and some notions related to convex analysis in Hadamard spaces, are given. Section \ref{flat} devoted to the notion of flat Hadamard spaces. A characterization result together with some new examples are considered. The  most important goal of section \ref{p-Fenchel} is to define the concept of $p$-Fenchel conjugate and prove the $p$-Fenchel-Young inequality. Finally, section \ref{p-Fitzpatrick transform} is devoted to prove a representation result for maximal monotone operators in terms of certain proper, convex,  l.s.c. extended real-valued functions on $X\times X^{\scalebox{0.7}{$^{\lozenge}$}}$. In this section, the notion of $p$-Fitzpatrick transform for a set-valued map from a Hadamard space $X$ to its linear dual $X^{\scalebox{0.7}{$^{\lozenge}$}}$ is introduced and considered. Then some characterization results for maximal monotonicity of set-valued maps from a Hadamard space $X$ to its linear dual $X^{\scalebox{0.7}{$^{\lozenge}$}}$ are given.

\newcommand{\Lip}{\mathrm{Lip}}
\newcommand{\Ran}{\mathrm{Range}}
\newcommand{\Dom}{\mathrm{Dom}}
\newcommand{\dom}{\mathrm{dom}}
\newcommand{\gra}{\mathrm{gra}}
\newcommand{\spa}{\mathrm{span}}
\newcommand{\langg}{\mathrm{\langle\!\langle}}
\newcommand{\rangg}{\mathrm{\rangle\!\rangle}}
\newcommand{\Fix}{\mathrm{Fix}}
\newcommand{\co}{\mathrm{co}}
\newcommand{\Id}{\mathrm{Id}}
\newcommand{\argmin}{\mathrm{argmin}}
\newcommand{\epi}{\mathrm{epi}}
\newcommand{\Int}{\mathrm{Int}}
\def\N {{\mathbb N}}
\def\chap {{\langle\!\!\langle}}
\def\R {{\mathbb R}}
\newcommand*{\shp}{{\scalebox{0.55}{$\#$}}}
\def\ssmath#1{\text{\scalebox{0.7}{$#1$}}}
\def\ssfrac#1#2{\ssmath{\frac{#1}{#2}}}
\def\smath#1{\text{\scalebox{0.8}{$#1$}}}
\def\sfrac#1#2{\smath{\frac{#1}{#2}}}
\def\mmath#1{\text{\scalebox{0.95}{$#1$}}}
\def\mfrac#1#2{\mmath{\frac{#1}{#2}}}
\allowdisplaybreaks

\section{Preliminaries}\label{pre}

\vspace*{-0.15cm}
 As a prelude to this paper, we present some preliminary and basic results about geodesic and Hadamard spaces. More details can be found in  \cite{BridsonHaefliger,Bacak2014,Papadupolos,KakavandiAmini,ChaipunyaKumam} and the references cited therein.

A metric space $(X,d)$  is said to be a \textit{geodesic space} if any two points of it can be connected by a geodesic path. Recall that an isometry $c:[0,1] \rightarrow X$ with $c(0)=x$ to $c(1)=y$ is a \textit{geodesic path} from $x$ to $y$. Any geodesic space $(X,d)$ satisfying the following \textit{CN-inequality} is called a \textit{$\cat(0)$ space},
\begin{equation}\label{cn}
d(z,c(t))^2 \leq (1-t)d(z,c(0))^2 +td(z,c(1))^2-t(1-t)d(c(0),c(1))^2,
\end{equation}
for each $x,y,z \in X$, $t \in [0,1]$ and $c$, a geodesic path from $x$ to $y$.

From now on, in a $\cat(0)$ space $(X,d)$, we set $(1-t)x \oplus ty:=c(t)$, where $x,y \in X$, $t \in [0,1]$ and $c:[0,1] \rightarrow X$ is a geodesic path. Moreover, $[x,y]$  denote  the image of the geodesic path $c$. Indeed,
\[[x,y]=\big\{(1-t)x \oplus ty: t \in [0,1]\big\}.\]
We use the term \textit{Hadamard space} for a complete $\cat(0)$  space.

The notion of the \textit{dual space} of a Hadamard space $(X,d)$ has been introduced in \cite{KakavandiAmini}, which is based on the \textit{quasilinearization} of an abstract metric space in \cite{BergNikolaev}.

Let $(X,d)$ be a Hadamard space, for each $x, y\in X$,  $\overrightarrow{xy}:=(x, y)$ is called a bound vector. The symbol $\mathbf{0}_x:=\overrightarrow{xx}$ will denotes the \textit{zero bound vector} at $x\in X$. We identify two bound vectors $-\overrightarrow{xy}$ and $\overrightarrow{yx}$. In $X^2$, two bound vectors $(x, y)$ and
$(u, v)$ are said to be \textit{compatible} if  $y=u$, in this case, we define
 \[\overrightarrow{xu}+\overrightarrow{uv}:=\overrightarrow{xv}.\]
The \textit{quasilinearization map} is defined by
\begin{align}\label{inb}
\langle \cdot,\cdot\rangle &: X^2 \times X^2 \rightarrow \mathbb{R},\\ \nonumber
\langle \overrightarrow{xy}, \overrightarrow{uv} \rangle & :=\sfrac{1}{2}\big(d(x,v)^2+d(y,u)^2-d(x,u)^2-d(y,v)^2\big), ~~(x,y,u,v,\in X).
\end{align}

Throughout this paper, $\mathcal{H}$ stands for a real Hilbert space  $(\mathcal{H},\langg{\cdot}\,|\,{\cdot}\rangg)$, where  $\langg{\cdot}\,|\,{\cdot}\rangg$ is an  inner product on  $\mathcal{H}$.
\begin{remark}\cite{KakavandiAmini,BergNikolaev}\label{4mquasilintoinpro}
It can be shown that the quasilinearization map of a real Hilbert space $(\mathcal{H},\langg{\cdot}\,|\,{\cdot}\rangg)$ as a $\cat(0)$ space, is the same as its inner product. Indeed, for each $a,b,c,d \in \mathcal{H}$,
\begin{align*}
\langle \overrightarrow{ab},\overrightarrow{cd} \rangle &=\sfrac{1}{2}\big(\|a-d\|^2+\|b-c\|^2-\|a-c\|^2-\|b-d\|^2\big)\\
&=\sfrac{1}{2}\big(\|a\|^2 -2\langg{a\,}|\,{d}\,\rangg + \|d\|^2 + \|b\|^2 -2\langg{b\,}|{\,c}\rangg + \|c\|^2\\
&\hspace{0.4cm}- \|a\|^2 +2\langg{a\,}|{\,c}\rangg - \|c\|^2-\|b\|^2 +2\langg{b\,}|{\,d}\rangg - \|d\|^2\big)\\
&=\langg{b-a\,}|{\,d-c}\rangg,
\end{align*}
as claimed.
\end{remark}
A metric space $(X,d)$ satisfies the \textit{Cauchy-Schwarz inequality} if 
\begin{equation} \label{csineq}
\langle \overrightarrow{xy}, \overrightarrow{uv} \rangle \leq d(x,y)d(u,v),~\forall x,y,u,v\in X,
\end{equation}
or equivalently,
\begin{equation*}
|\langle \overrightarrow{xy}, \overrightarrow{uv} \rangle| \leq d(x,y)d(u,v).
\end{equation*}

It is known that a geodesic metric space $(X,d)$ is a $\cat(0)$ space if and only if it satisfies the Cauchy-Schwarz inequality (see {\rm \cite[Corollary 3]{BergNikolaev}}).

Let $C(X,\mathbb{R})$ (resp. $\Lip(X,\mathbb{R})$) denotes the set of all continuous (resp. Lipschitz) real-valued functions on $X$.
Given a Hadamard space $(X,d)$, define the map (\cite{KakavandiAmini})
\begin{align*}
\Theta : &\mathbb{R} \times X^2 \rightarrow C(X,\mathbb{R})\\
& (t,a,b) \mapsto \Theta (t,a,b)x=t \langle \overrightarrow{ab}, \overrightarrow{ax} \rangle, ~~(a,b,x \in X , t \in \mathbb{R}).
\end{align*}
It follows from the Cauchy-Schwarz inequality that $\Theta (t,a,b)$ is a Lipschitz function with the Lipschitz semi-norm
\begin{equation}\label{lips}
L(\Theta (t,a,b) )=|t|d(a,b) ,~~(a,b \in X,~t \in \mathbb{R}).
\end{equation}
Recall that the \textit{Lipschitz semi-norm} $L:\Lip(X,\mathbb{R})\rightarrow \mathbb{R}$ is defined by
\begin{align*}
L(\varphi):=&\sup\Big\{ \frac{|\varphi(x)-\varphi(y)|}{d(x,y)}:x,y \in X, x\neq y\Big\}\\
=&\sup\Big\{ \frac{\varphi(x)-\varphi(y)}{d(x,y)}:x,y \in X, x\neq y\Big\}.
\end{align*}
The \textit{pseudometric} $D$ on $\mathbb{R} \times X^2$, induced by  Lipschitz semi-norm (\ref{lips}), is defined by
\begin{align}\label{dinnerd}
D\big((t,a,b),(s,c,d)\big)=L\big(\Theta (t,a,b)-\Theta (s,c,d)\big),~~(a,b,c,d \in X, \,t,s \in \mathbb{R}).
\end{align}
Recall from  \cite{KakavandiAmini} that the pseudometric space $(\mathbb{R} \times X^2,D)$ can be viewed as a subspace of the pseudometric space of all real-valued Lipschitz functions $(\Lip(X,\mathbb{R}), L)$.
\begin{lemma}{\rm\cite[Lemma 2.1]{KakavandiAmini}} Let $X$ be a Hadamard space and $a,b, c, d\in X$. Then
\[ D\big((t,a,b),(s,c,d)\big)=0 ~\Longleftrightarrow ~\big[t\langle \overrightarrow{ab},\overrightarrow{xy} \rangle =s\langle \overrightarrow{cd},\overrightarrow{xy} \rangle,\,\forall\, t, s \in \mathbb{R}, \forall\, x, y\in X\big],\]
where $D$ is as in {\rm(\ref{dinnerd})}.
\end{lemma}
It is easy to check that the following relation on $\mathbb{R} \times X^2$ is an equivalence relation
\[(t,a,b)\sim(s,c,d):\Leftrightarrow D\big((t,a,b),(s,c,d)\big)=0.\]
Therefore, for every $(t,a,b) \in\mathbb{R} \times X^2 $, the equivalence class of $(t,a,b)$ is given by
\[[t\overrightarrow{ab}]:=\big\{s\overrightarrow{cd}: D\big((t,a,b),(s,c,d)\big)=0, s\in\mathbb{R}, c, d\in X\big\}.\]
By the definition of equivalence classes we have $[\overrightarrow{aa}]=[\overrightarrow{bb}]$ for each $a,b \in X$. Now, the set $X^{_{\scalebox{1.025}{*}}}:=\big\{[t\overrightarrow{ab}]:a,b \in X,t \in \mathbb{R}\big\}$  is called the \textit{dual space} of  $X$. The zero element $\mathbf{0}_{X^{_{\scalebox{1.015}{*}}}}:=[t\overrightarrow{aa}]$, is briefly denoted by $\mathbf{0}$, where  $a \in X$ and $t \in \mathbb{R}$. Moreover, for each $a,b \in X$, we have $\langle \mathbf{0}_{X^{_{\scalebox{1.015}{*}}}}, \overrightarrow{ab} \rangle$=0.
 It is easily seen that $X^{_{\scalebox{1.025}{*}}}$ is a metric space equipped with the metric $D$, defined by
\[D\big([t\overrightarrow{ab}], [s\overrightarrow{cd}]\big):=D\big((t,a,b),(s,c,d)\big),~~( t,s \in \mathbb{R},\,a,b,c,d, \in X).\]
Furthermore, $X^{_{\scalebox{1.025}{*}}}$ acts on $X^2$ by
\begin{equation}\label{efxonab}
\langle x^*,\overrightarrow{xy} \rangle:=t \langle \overrightarrow{ab}, \overrightarrow{xy} \rangle,~\text{where}~x^*=[t\overrightarrow{ab}] \in X^{_{\scalebox{1.05}{*}}}~\text{and}~ \overrightarrow{xy} \in X^2.
\end{equation}
Moreover, we use the following notation:
\begin{equation}\label{efxlozonab1}
\langle \alpha x^*+\beta y^*,\overrightarrow{xy} \rangle:=\alpha \langle x^*,\overrightarrow{xy} \rangle + \beta \langle y^*,\overrightarrow{xy} \rangle,
\end{equation}
where $x, y\in X$, $x^*,y^* \in X^{_{\scalebox{1.025}{*}}}$ and $\alpha, \beta \in \mathbb{R}$.

The concept of the \textit{linear dual space} of a Hadamard space $X$, denoted by $X^\smalllozenge$, was introduced in \cite{ChaipunyaKumam} as follows
\[X^\smalllozenge:=\Big\{\sum_{i=1}^n\alpha_i x_i^{_{\scalebox{0.65}{$*$}}} :\alpha_i \in \mathbb{R}, x_i^{_{\scalebox{0.65}{$*$}}} \in X^*,\,1\leq i \leq n,\, n \in \mathbb{N}\Big\}.\]

Throughout this paper, we assume that $p \in X$ is arbitrary and fixed. Note that
\begin{align*}
X^\smalllozenge=\Big\{\sum^n_{i=1}\alpha_i[\overrightarrow{pa_i}]:\alpha_i\in \mathbb{R}, a_i\in X,\,1\leq i \leq n,\, n \in \mathbb{N}\Big\}.
\end{align*}
From (\ref{efxlozonab1}), $X^\smalllozenge$ is a real vector space.
It is known from \cite[Proposition 3.5]{ZamaniRaeisi} that
\begin{align*}
\|x^\loz \|_{\loz}:=\sup\Big\{\sfrac{\big|\langle x^\loz,\overrightarrow{ab}\rangle-\langle x^\loz,\overrightarrow{cd}\rangle\big|}{d(a,b)+d(c,d)}: a,b,c,d \in X, (a,c)\neq (b,d) \Big\},~~(x^\loz\in X^\smalllozenge),
\end{align*}
is a norm on $X^\smalllozenge$.

As it is proved in \cite[Remark 1.3.]{MR0}, for each $a,b,w \in X$ and each 
$x^\loz\!\!\in\! X^\smalllozenge\!$, we have:
\begin{align*}
\langle x^\loz,\overrightarrow{ab} \rangle
=\langle x^\loz,\overrightarrow{aw} \rangle+\langle x^\loz,\overrightarrow{wb} \rangle.
\end{align*}

In the sequel, the following notations and definitions will be used.

The set of all extended real numbers is denoted by $\overline{\mathbb{R}}:=[-\infty,\infty]$. The \textit{indicator function} of $A \subseteq X$ is defined as
\begin{align*}
\iota_A&:X \rightarrow \overline{\mathbb{R}}\\
\iota_A(x)&:=\begin{cases}
0& x \in A,
\\[1mm]
\infty& x \notin A.
\end{cases}
\end{align*}
\begin{definition}{\rm \cite[Section 2.2]{Bacak2014}\label{lscepi} Let $(X,d)$ be a Hadamard space and $f : X \rightarrow ]-\infty,\infty]$ be a function. Then
\begin{enumerate}\setlength\itemsep{0.05em}
\item[(i)] the \textit{domain} of $f$ is defined by
\[\dom f:=\{x \in X : f(x) <\infty \}.\]
Moreover, $f$ is said to be \textit{proper} if $\dom f\neq \emptyset$.
\item[(ii)] $f$ is \textit{lower semi-continuous} (\textit{l.s.c.}, for short) if for each $\alpha \in \mathbb{R}$, the \textit{epigraph of $f$}, i.e., the set $\epi f:=\{(x, \alpha)\in X\times\mathbb{R} : f(x) \leq \alpha \}$, is closed.
\item[(iii)]  $f$ is called \textit{convex} if for each $x,y \in X$ and each $\lambda \in [0,1]$,
\[f\big((1-\lambda)x \oplus \lambda y\big) \leq (1-\lambda)f(x)+\lambda f(y).\]
\end{enumerate}
}\end{definition}
The set of all proper, l.s.c. and convex extended real-valued functions on $X$ is denoted by $\Gamma(X)$.
 Letting $C \subseteq X \times X^\smalllozenge$, we say that $C$ is \textit{convex} if  \[\big((1-\lambda)x \oplus \lambda y,(1-\lambda)x^\loz+\lambda y^\loz\big) \in C,\]
 for each $(x,x^\loz),\,(y,y^\loz) \in C$ and each $\lambda \in [0,1]$.

Let $C$ be a convex subset of $X \times X^\smalllozenge$. The function $h :X\times X^\smalllozenge \rightarrow\overline{\R}$ is called \textit{convex} on $C$ if for each $(x,x^\loz),\,(y,y^\loz) \in C$ and each $\lambda \in [0,1]$,
\[h\big((1-\lambda)x\oplus \lambda y,(1-\lambda)x^\loz+\lambda y^\loz\big) \leq (1-\lambda)h(x,x^\loz)+\lambda h(y,y^\loz).\]
In addition, $h$ is said to be \textit{convex} if $h$ is convex on $X\times X^\smalllozenge$.
Moreover, $h$ is \textit{proper} if $-\infty\notin h(X\times X^\smalllozenge)$ and $h \not\equiv\infty$. Set\vspace{-0.05cm}
\[\Gamma(X\times X^\smalllozenge):=\big\{h :X\times X^\smalllozenge \rightarrow\overline{\mathbb{R}}~|~ h\,\text{ is proper, l.s.c. and convex}\big\}.\vspace{-0.05cm}\]
The \textit{$p$-coupling function} of the dual pair $(X,X^\smalllozenge)$ is defined by:\vspace{-0.05cm}
 \[\pi_p:X \times X^\smalllozenge \rightarrow \mathbb{R},\, (x, x^\loz)\mapsto \langle x^\loz, \overrightarrow{px} \rangle,~~\big((x, x^\loz)\in X \times X^\smalllozenge\big).\vspace{-0.05cm}\]
 %Note that, this mapping is useful for simplifying many results of this paper.
Some properties of the $p$-coupling function can be found in \cite{MRMA2}.
Define the mapping\vspace{-0.05cm}
\[r:X\times X^\smalllozenge\rightarrow X^\smalllozenge\times X,\,(x, x^{\loz})\mapsto (x^{\loz}, x), ~\big((x, x^\loz)\in X \times X^\smalllozenge\big).\vspace{-0.05cm}\]

Given  any $h : X\times X^\smalllozenge\rightarrow\overline{\mathbb{R}}$ and $p\in X$, for $\divideontimes \in \{\leq,=,\geq,<,>\}$ set
\begin{align*}
 \{ h\divideontimes\pi_p\}&:=\{(x,x^\loz) \in X\times X^\smalllozenge: h(x,x^\loz) \divideontimes \pi_p(x, x^\loz)\}.
 \end{align*}

\section{Flat Hadamard space}\label{flat}
In the remaining parts of this paper, $X$ is a Hadamard space with the linear dual space $X^\smalllozenge$, unless otherwise specified.
\begin{definition}\rm{
The \textit{domain} and \textit{range} of $M\subseteq X\times X^\smalllozenge$ are defined, respectively,  by\vspace{-0.05cm}
\begin{equation*}
\Dom(M):=\{ x\in X : \exists\, x^\loz\in X^\smalllozenge ~\text{s.t.}~ (x, x^\loz)\in M\},\vspace{-0.15cm}
\end{equation*}
and\vspace{-0.05cm}
\begin{equation*}
\Ran(M):=\{ x^\loz\in X^\smalllozenge : \exists\, x\in X ~\text{s.t.}~ (x, x^\loz)\in M\}.\vspace{-0.05cm}
\end{equation*}
}\end{definition}
\begin{definition}{\rm\cite[Definition 3.2]{MRMA2}}\label{zsd}{\rm\, We say that $ M \subseteq X \times X^\smalllozenge$ satisfies
\begin{enumerate}\setlength\itemsep{0.05em}
\item[(i)] \textit{$\mathcal{F}_l$-property} if for each $\lambda \in [0,1]$, $x^\loz \in\Ran(M)$ and $x,y \in\Dom(M)$;
\begin{equation}\label{EqWp}
\big \langle x^\loz,\overrightarrow{p((1-\lambda)x\oplus \lambda y)} \big \rangle\leq(1-\lambda)\langle x^\loz,\overrightarrow{px}\rangle +\lambda \langle x^\loz,\overrightarrow{py} \rangle.
\end{equation}
\item[(ii)] \textit{$\mathcal{F}_g$-property} if for each $\lambda \in [0,1]$, $x^\loz \in\Ran(M)$ and $x,y \in\Dom(M)$;
\begin{equation}\label{EqWl}
\big \langle x^\loz,\overrightarrow{p((1-\lambda)x\oplus \lambda y)} \big \rangle\geq(1-\lambda)\langle x^\loz,\overrightarrow{px}\rangle +\lambda \langle x^\loz,\overrightarrow{py} \rangle.\end{equation}
\item[(iii)] \textit{$\mathcal{F}$-property} if it has both $\mathcal{F}_l$-property and $\mathcal{F}_g$-property.
\end{enumerate}
}\end{definition}

Let $\mathcal{P}\in\{\mathcal{F}_l, \mathcal{F}, \mathcal{F}_g\}$. It follows from {\rm{\rm\cite[Remark 3.3(ii)]{MRMA2}}}  that the $\mathcal{P}$-property is independent of the choice of the point $p\in X$.

Recall that \cite[Definition 3.1]{LurieNotes} a Hadamard space $(X,d)$ is called \textit{flat} if equality holds in the CN-inequality \eqref{cn}. Some examples of flat Hadamard spaces are Hilbert spaces and closed unit balls in any Hilbert space.

As we will see in the next theorem, the $\mathcal{P}$-property provides a characterization of flatness. Indeed,
\begin{theorem}{\rm{\rm\cite[Theorem 3.5]{MRMA2}}}  \label{frtg}The following statements for a Hadamard space $X$ with the linear dual $X^\smalllozenge$ are equivalent.
\begin{enumerate}\setlength\itemsep{0.05em}
\item[{\rm(i)}] $X$  is flat.

\item[{\rm(ii)}] $X\times X^\smalllozenge$ has the $\mathcal{P}$-property.

\item[{\rm(iii)}] Any subset of $X\times X^\smalllozenge$ has the $\mathcal{P}$-property.
 \end{enumerate}
  \end{theorem}

\begin{remark}\label{AK01}
(i) It is known \cite[Definition and Notation 2.2.]{KakavandiAmini} that, if $X$ is a closed and convex subset of a Hilbert space $(\mathcal{H},\langg{\cdot}\,|\,{\cdot}\rangg)$ with a non-empty interior, then $X^*=\mathcal{H}$. For the sake of completeness, we add a modified version of the proof. Let $a \in \Int(X)$. Hence, $B_{\epsilon}[a]:=\{u\in\mathcal{H}: \|u-a\|\leq\epsilon\}\subseteq X$ for some $\epsilon >0$. Now, we define the map $j_a:\mathcal{H} \rightarrow X^*$ by
\begin{align}\label{jpdef}
j_a(x)&:=\begin{cases}
\Big[\frac{\|x\|}{\epsilon}\overrightarrow{a\big(a+\frac{\epsilon x}{\|x\|}\big)}\Big]& x \neq 0 ,
\\[1mm]
 \mathbf{0}& x=0 .
\end{cases}
\end{align}
It turns out that $j_a$ is a surjective isomorphism. To see this, first we consider the following cases:

\textbf{Case 1:~}Let $x,y \in \mathcal{H}\setminus \{0\}$, set $\alpha:=\frac{\|x\|}{\epsilon}$ and $\beta:=\frac{\|y\|}{\epsilon}$. Then
\begin{align*}
\hspace{-0.35cm}&D(j_a(x),j_a(y))=L\big(\Theta\big(\alpha,a,a+\sfrac{x}{\alpha}\big)-\Theta\big(\beta,a,a+\sfrac{y}{\beta}\big)\big)\\[-0.5mm]
&=\!\!\!\sup_{u,v \in X, \,u\neq v}\!\!\!\!\!\!\frac{\big(\Theta\big(\alpha,a,a\!+\!\mfrac{x}{\alpha}\big)\!-\!\Theta\big(\beta,a,a\!+\!\mfrac{y}{\beta}\big)\big)(u)\!-\!\big(\Theta\big(\alpha,a,a\!+\!\mfrac{x}{\alpha}\big)\!-\!\Theta\big(\beta,a,a\!+\!\mfrac{y}{\beta}\big)\big)(v)}{\|u-v\|}\\[-0.5mm]
&=\sup_{u,v \in X, \,u\neq v}\frac{\alpha\big(\langle\overrightarrow{a(a+\mfrac{x}{\alpha})},\overrightarrow{au}\rangle\!-\!\langle\overrightarrow{a(a+\mfrac{x}{\alpha})},\overrightarrow{av}\rangle\big)+\beta\big(\langle\overrightarrow{a(a+\mfrac{y}{\beta})},\overrightarrow{av}\rangle\!-\!\langle\overrightarrow{a(a+\mfrac{y}{\beta})},\overrightarrow{au}\rangle\big)}{\|u-v\|}\\[-0.5mm]
&=\sup_{u,v \in X, \,u\neq v}\frac{\alpha\langle\overrightarrow{a(a+\mfrac{x}{\alpha})},\overrightarrow{vu}\rangle+\beta\langle\overrightarrow{a(a+\mfrac{y}{\beta})},\overrightarrow{uv}\rangle}{\|u-v\|}\\[-0.5mm]
&=\sup_{u,v \in X, \,u\neq v}\frac{\alpha\langg{\frac{x}{\alpha} \,}|{ \,u-v}\rangg+\beta\langg{\sfrac{y}{\beta} \,}|{ \,v-u}\rangg}{\|u-v\|}\,\qquad\qquad\qquad\qquad~~\qquad\hfill (\text{by Remark \ref{4mquasilintoinpro}})\\[-0.5mm]
&=\sup_{u,v \in X, \,u\neq v}\frac{\langg{x-y \,}|{ \,u-v}\rangg}{\|u-v\|}\\[-0.5mm]
&=\|x-y\|.
\end{align*}
\textbf{Case 2:~} Let $x \in \mathcal{H}\setminus \{0\}$, $y=0$ and set $\alpha:=\frac{\|x\|}{\epsilon}$. Then
\begin{align*}
D(j_a(x),j_a(y))&=D(j_a(x),j_a(0))\\&=L\big(\Theta\big(\alpha,a,a+\ssfrac{x}{\alpha}\big)\big)\\[-0.5mm]
&=\sup_{u,v \in X,\,u\neq v}\frac{\Theta\big(\alpha,a,a+\mfrac{x}{\alpha}\big)(u)-\Theta\big(\alpha,a,a+\mfrac{x}{\alpha}\big)(v)}{\|u-v\|}\\[-0.5mm]
&=\sup_{u,v \in X,\,u\neq v}\frac{\alpha\big(\langle\overrightarrow{a(a+\mfrac{x}{\alpha})},\overrightarrow{au}\rangle-\langle\overrightarrow{a(a+\mfrac{x}{\alpha})},\overrightarrow{av}\rangle\big)}{\|u-v\|}\\[-0.5mm]
&=\sup_{u,v \in X,\,u\neq v}\frac{\alpha\langle\overrightarrow{a(a+\mfrac{x}{\alpha})},\overrightarrow{vu}\rangle}{\|u-v\|}\\[-0.5mm]
&=\sup_{u,v \in X,\,u\neq v}\frac{\alpha\langg{\mfrac{x}{\alpha}\,}|{\,u-v}\rangg}{\|u-v\|}\,\qquad \qquad\qquad\qquad~\quad(\text{using Remark \ref{4mquasilintoinpro}})\\[-0.5mm]
&=\|x\|.
\end{align*}
To prove the surjectivity of $j_a$, let $[t\overrightarrow{xy}] \in X^*\setminus \{\mathbf{0}\}$ and $u,v \in X$ be arbitrary, using Remark \ref{4mquasilintoinpro} we get:
\begin{align*}
\langle j_a(t(y-x)), \overrightarrow{uv} \rangle &=\frac{|t|\|y-x\|}{\epsilon}\Bigg\langle \overrightarrow{a\Big(a+\frac{\epsilon t(y-x)}{\|t(y-x)\|}\Big)}, \overrightarrow{uv} \Bigg\rangle\\
&=t\langg{y-x\,}|{\, v-u}\rangg\\
&=\langle [t\overrightarrow{xy}],\overrightarrow{uv}\rangle,
\end{align*}
which implies that
\begin{equation}\label{Hassan01}
j_a(t(y-x))=[t\overrightarrow{xy}].
\end{equation}

(ii) As a special case of (i), let $X=\mathcal{H}$.  Using \eqref{Hassan01},
we obtain that $j_a(x)=\big[\overrightarrow{0x}\big]$ for all $a, x \in \mathcal{H}$. In particular, $j_0(x)=\big[\overrightarrow{0x}\big]$ for all $x \in \mathcal{H}$. Therefore, in the sense of isomorphism, one can write $\mathcal{H}^*=\mathcal{H}^\smalllozenge=\mathcal{H}$.
\end{remark}
\begin{remark}\label{Rem-Flat}
Note that any nonempty, closed and convex subset of a flat Hadamard space is also flat and so is any nonempty, closed and  convex subset of a Hilbert space. Moreover, according to \cite[Theorem 7.2]{LurieNotes}, any flat Hadamard space is isometric to a nonempty, closed and convex subset of a Hilbert space. Also, by
\cite[Theorem 3.3]{MovahediBehmardiSoleimani-Damaneh}, any flat Hadamard real vector space with a
translation invariant metric is a Hilbert space. It is worth noting that any closed unit ball in arbitrary Hilbert spaces is a flat Hadamard space, which is not even a vector space
(see \cite[Remark 3.4]{MovahediBehmardiSoleimani-Damaneh}).
\end{remark}

The following examples will be used in the sequel.
\begin{example}\label{ExMonS}
Consider the following equivalence relation on $\mathbb{N} \times [0,1]$:
\[(n,t)\sim (m,s):\Leftrightarrow t=s=0 ~\text{or} ~(n,t)=(m,s).\]
Set $ X:=\frac{\mathbb{N} \times [0,1]}{\sim} $ and let $ d :X \times X \rightarrow \mathbb{R} $ be defined by
\begin{align*}
d([(n,t)],[(m,s)])=\begin{cases}
|t-s| &n=m,
\\
t+s  & n\neq m.
\end{cases}
\end{align*}
The geodesic joining $x=\big[(n,t)\big]$ to $y=\big[(m,s)\big] $ is defined as follows:
\[ (1-\lambda)x \oplus \lambda y:=\begin{cases}
\big[(n,(1-\lambda)t-\lambda s)\big] & 0 \leq \lambda \leq \frac{t}{t+s},
\\[1mm]
\big[(m,(\lambda-1)t+\lambda s)\big]   & \frac{t}{t+s} \leq  \lambda \leq 1,
\end{cases}
\]
whenever $x\neq y$ and otherwise  $(1-\lambda)x \oplus \lambda x:=x$ for each $\lambda\in[0,1]$.

It is well known that (see \cite[Example 4.7]{Kakavandi}) $(X,d)$ is an $\mathbb{R}$-tree space. It follows from \cite[Example 1.2.10]{Bacak2014}, that $\mathbb{R}$-tree spaces are Hadamard spaces. Moreover, we see from  \cite[Example 3.7]{MRMA2} that $(X,d)$ is a non-flat Hadamard space.
\end{example}

\begin{example} Consider $\mathbb{R}^{n+1}$ equipped with the \textit{$(-1,n)$-inner product}
\[\langle x,y \rangle_{(-1,n)}:=-x_{n+1}y_{n+1}+\sum_{i=1}^n x_iy_i ,~ x=(x_1,\ldots,x_{n+1}), \,y=(y_1,\ldots,y_{n+1})\in \mathbb{R}^{n+1}.\]
The \textit{real hyperbolic $n$-space} is defined by \[\mathbb{H}^n:=\big\{ x=(x_1,\ldots,x_{n+1})\in \mathbb{R}^{n+1}:\langle x,x \rangle_{(-1,n)}=-1,\, x_{n+1}>0\big\}.\]
 Define the metric $d:\mathbb{H}^n \times \mathbb{H}^n \rightarrow [0,\infty)$ by \[d(x,y):=\cosh^{-1}\big(-\langle x,y \rangle_{(-1,n)}\big).\] In addition, the geodesic path from $x\in \mathbb{H}^n$  to $y \in \mathbb{H}^n$  is defined by
\begin{align*}
(1-t)x \oplus t y:= \big(\cosh td(x,y)\big)x+\big(\sinh td(x,y)\big)z,~~(t \in [0,1]),
\end{align*}
where, $z:=\frac{y+\langle x,y \rangle_{(-1,n)}x}{\sqrt{\langle x,y \rangle_{(-1,n)}^2-1}}$.
It is known that $\mathbb{H}^n$ is a Riemannian $n$-manifold with the sectional curvature $-1$ at any point. In particular, $\mathbb{H}^n$ is a Hadamard space. For more details see \cite[Example 1.2.11]{Bacak2014}. We claim that $\mathbb{H}^n$ is a non-flat Hadamard space. To see this, let $x^*:=[\overrightarrow{ab}]$, where
\[a:=\big(1,0,\ldots,0,\sqrt{2}\big) \text{ and } b:=\big(-1,0,\ldots,0,\sqrt{2}\big). \]
Assume that  \[x:=\big(0,1,0,\ldots,0,\sqrt{2}\big) \text{ and } y:=\big(-1,0,\ldots,0,\sqrt{2}\big). \] One can see that $\langle x,y \rangle_{(-1,n)}=-2$.
Then
\[z=\frac{y+\langle x,y \rangle_{(-1,n)}x}{\sqrt{\langle x,y \rangle_{(-1,n)}^2-1}}=\frac{1}{\sqrt{3}}(y-2x)=\frac{1}{\sqrt{3}}\big(-1,-2,0,\ldots,0,-\sqrt{2}\big).\]
Moreover, $d(x,y)=\cosh^{-1}\big(-\langle x,y \rangle_{(-1,n)}\big)=\cosh^{-1}(2)$. Set $t:=\frac{1}{2}$, then
\begin{align*}
(1-t)x \oplus ty&=\sfrac{1}{2}x \oplus \sfrac{1}{2} y\\
&=\big(\cosh \sfrac{1}{2}d(x,y)\big)x+\big(\sinh \sfrac{1}{2}d(x,y)\big)z\\
&=\Big(\cosh \sfrac{1}{2}\big(\cosh^{-1}(2)\big)\Big)x+\Big(\sinh \sfrac{1}{2}\big(\cosh^{-1}(2)\big)\Big)z\\
&=\alpha(0,1,0,\ldots,0,\sqrt{2})+\beta (-1,-2,0,\ldots,0,-\sqrt{2})\\
&=(-\beta,\alpha-2\beta,0,\ldots,0,\sqrt{2}(\alpha-\beta)).
\end{align*}
where $\alpha:=\cosh\big(\frac{1}{2}\cosh^{-1} 2\big)$ and $\beta:=\frac{\sinh \big(\frac{1}{2}\cosh^{-1} 2\big)}{\sqrt{3}}$.
Therefore,
\begin{align*}
\langle x^*,\overrightarrow{x((1-t)x \oplus ty)} \rangle &=\Big\langle [\overrightarrow{ab}],\overrightarrow{x\big(\ssfrac{1}{2}x \oplus \ssfrac{1}{2} y\big)}\Big\rangle=\Big\langle \overrightarrow{ab},\overrightarrow{x\big(\ssfrac{1}{2}x \oplus \ssfrac{1}{2}y\big)}\Big\rangle\\
 &=\sfrac{1}{2}\Big(\big(\cosh^{-1}(2\alpha-\beta)\big)^2+\big(\cosh^{-1}(2)\big)^2\\&~~~~~~~~~-\big(\cosh^{-1}(2)\big)^2-\big(\cosh^{-1}(2\alpha-3\beta)\big)^2\Big)\\
  &=\sfrac{1}{2}\Big(\big(\cosh^{-1}(2\alpha-\beta)\big)^2-\big(\cosh^{-1}(2\alpha-3\beta)\big)^2\Big)\\
  &\cong 0.6816.
\end{align*}
On the other hand,
\begin{align*}
\sfrac{1}{2}\langle x^*,\overrightarrow{xx} \rangle+\sfrac{1}{2}\langle x^*,\overrightarrow{xy} \rangle &=\sfrac{1}{2}\langle x^*,\overrightarrow{xy} \rangle =\sfrac{1}{2}\langle [\overrightarrow{ab}],\overrightarrow{xy} \rangle =\sfrac{1}{2}\langle \overrightarrow{ab},\overrightarrow{xy} \rangle \\
&=\sfrac{1}{2}\bigg(\sfrac{1}{2}\Big(\big(\cosh^{-1}(3)\big)^2+\big(\cosh^{-1}(2)\big)^2\\&~~~~~~~~~~-\big(\cosh^{-1}(2)\big)^2-\big(\cosh^{-1}(1)\big)^2\Big)\bigg)\\
&\cong 0.7768.
\end{align*}
Therefore,
\[\Big\langle x^*,\overrightarrow{x\big(\ssfrac{1}{2}x \oplus \ssfrac{1}{2} y\big)}\Big\rangle \neq \ssfrac{1}{2}\langle x^*,\overrightarrow{xx} \rangle+\ssfrac{1}{2}\langle x^*,\overrightarrow{xy} \rangle.\]
Hence, Theorem \ref{frtg} implies that $\mathbb{H}^n$ is  not flat.
\end{example}

\section{$p$-Fenchel conjugate}\label{p-Fenchel}
In this section, we introduce and investigate the concept of $p$-Fenchel conjugate and we prove some useful results.
%In this section, ...
\begin{definition} \label{$p$-Fenchel}
Let $h :X\times X^\smalllozenge\rightarrow\overline{\mathbb{R}}$. The \textit{$p$-Fenchel conjugate} of $h$ from $X^\smalllozenge\times X$ to $\overline{\R}$ is defined by:
\begin{align}
h_{p}^{\loz}(x^{\loz}, x)=\sup_{(y,y^{\loz})\in X\times X^\smalllozenge}\{\langle x^{\loz}, \overrightarrow{py}\rangle + \langle y^{\loz}, \overrightarrow{px}\rangle-h(y, y^{\loz}) \}. \label{t22}
\end{align}
\end{definition}
Moreover, set
\begin{equation}\label{Gamma-Loz-p}
\Gamma^{\,^\loz}_{p}(X):=\big\{h \in \Gamma(X\times X^\smalllozenge):
h=\big(h+\iota_{\{h \leq \pi_p\}}\big)_{p}^{\loz}\circ r \text{~on~} X\times X^\smalllozenge\big\}.
\end{equation}

Let $\big(\mathcal{H},\langg{\cdot}\,|\,{\cdot}\rangg\big)$ be a real Hilbert space and $h :\mathcal{H}\times \mathcal{H}\rightarrow\overline{\mathbb{R}}$. We recall from \cite{EliasMartinez-Legaz} that the Fenchel conjugate of $h$ is
\begin{align*}
 h^\star:&\mathcal{H}\times \mathcal{H}\rightarrow\overline{\mathbb{R}}\\
&(u,x)\mapsto \sup_{(y, v)\in\mathcal{H}\times \mathcal{H}}\{\langg{u\,}|{\,y}\rangg+\langg{v\,}|{\,x}\rangg-h(y,v)\},~~u,x\in\mathcal{H}.
\end{align*}

\begin{remark}\label{AK02}
\begin{enumerate}
\item[(i)]
Let $X$ be a closed and convex subset of the Hilbert space $\big(\mathcal{H},\langg{\cdot}\,|\,{\cdot}\rangg\big)$ with $a \in \Int(X)$. Let $h :X\times X^\smalllozenge\rightarrow\overline{\mathbb{R}}$ be given. Then,
\begin{align*}
h_{p}^{\loz}&(x^{\loz}, x)=h_{p}^{\loz}(j_a(u), x)\qquad\,\,\quad\!\!\!\big(x^{\loz}=j_a(u) \text{ for some } u\in\mathcal{H} \text{ by surjectivity of } j_a\big)\\
&=\sup_{(y,y^{\loz})\in X\times X^\smalllozenge}\big\{\langle j_a(u), \overrightarrow{py}\rangle + \langle y^{\loz}, \overrightarrow{px}\rangle-h(y, y^{\loz}) \big\}\\
&=\sup_{(y,v)\in X\times \mathcal{H}}\big\{\langle j_a(u), \overrightarrow{py}\rangle + \langle j_a(v), \overrightarrow{px}\rangle-h(y, j_a(v)) \big\}\quad\big(\text{by surjectivity of } j_a\big)\\
&=\sup_{(y,v)\in X\times \mathcal{H}}\big\{\langg{u\,}|{\,y-p}\rangg+\langg{v\,}|{\,x-p}\rangg-h(y, j_a(v)) \big\}\\
&=\sup_{(y,v)\in X\times \mathcal{H}}\big\{\langg{u\,}|{\,y}\rangg+\langg{v\,}|{\,x}\rangg-\tilde{h}_p (y, v) \big\}-\langg{u\,}|{\,p}\rangg \\
&=\big(\tilde{h}_p\big)^\star(x,u)-\langg{u\,}|{\,p}\rangg,
\end{align*}
where $\tilde{h}_p (y, v):=\langg{v\,}|{\,p}\rangg+h(y, j_a(v))$.
\item[(ii)] As a special case of (i), when $X=\mathcal{H}$ and $p=0$,  we have:
\[h_0^{\loz}([\overrightarrow{0u}], x)=\big(\tilde{h}_0\big)^\star(u,x)-\langg{u\,}|{\,0}\rangg=\big(\tilde{h}_0\big)^\star(u,x).\]
Note that $\tilde{h}_0 (y, v)=h(y, j_a(v))$. By Remark \ref{AK01}(ii), one can write $h_0^{\loz}=h^\star$.
\end{enumerate}
\end{remark}

\begin{proposition}\label{FenchelYoungIneqth}
Let $h :X\times X^\smalllozenge\rightarrow\overline{\mathbb{R}}$ be proper. Then
\begin{enumerate}\setlength\itemsep{-0.01em}
\item[{\rm(i)}]
for each $(x, x^{\loz}),  (y, y^{\loz})\in X\times X^\smalllozenge$ we have:
\begin{align}\label{FenchelYoungIneq}
\hspace{-0.2cm}h(x, x^{\loz})+h_p^{\loz}(y^{\loz},y)\geq \langle y^{\loz}, \overrightarrow{px}\rangle+\langle x^{\loz}, \overrightarrow{py}\rangle. ~\textrm{(the {$p$-Fenchel-Young Inequality})}
\end{align}
\item[{\rm(ii)}] $\big\{\sfrac{1}{2}\big( h+h_{p}^{\loz}\circ r\big) \geq \pi_p\big\}=X \times X^\smalllozenge$.
\item[{\rm(iii)}] $\{h\geq\pi_p\}=X \times X^\smalllozenge$ for each $h \in \Gamma^{\,^\loz}_{p}(X)$.
\end{enumerate}
\end{proposition}
\begin{description}\setlength\itemsep{-0.01em}
\item[{\textit{\textbf{Proof.}}}~{\rm (i)}]  Let $(x, x^{\loz}), (y, y^{\loz})\in X \times X^\smalllozenge$.   Using Definition \ref{$p$-Fenchel} we obtain:
\begin{align*}
h_p^{\loz}(y^{\loz}, y)&=\sup_{(z,z^{\loz})\in X\times X^\smalllozenge}\{\langle y^{\loz}, \overrightarrow{pz}\rangle + \langle z^{\loz}, \overrightarrow{py}\rangle-h(z, z^{\loz}) \}\\
&\geq\langle y^{\loz}, \overrightarrow{px}\rangle + \langle x^{\loz}, \overrightarrow{py}\rangle-h(x, x^{\loz}),
\end{align*}
and hence \eqref{FenchelYoungIneq} holds.

\item[{\rm(ii)}] This is a direct consequence of the $p$-Fenchel-Young inequality \eqref{FenchelYoungIneq}.

\item[{\rm(iii)}] Let  $h \in \Gamma^{\,^\loz}_{p}(X)$. First assume that $(x,x^{\loz}) \notin \{h\leq\pi_p\} $. Then
\[h (x, x^\loz) > \langle x^{\loz}, \overrightarrow{px}\rangle.\]
On the other hand, for each $(x,x^{\loz}) \in \{h\leq\pi_p\} $, using $h \in \Gamma^{\,^\loz}_{p}(X)$ and (ii) we have:
\begin{align*}
h(x, x^{\loz})&=\sfrac{1}{2}\Big(h (x, x^{\loz})+\big(\big(h+\iota_{\{h \leq \pi_p\}}\big)_{p}^{\loz}\circ r\big)(x, x^{\loz})\Big)\\ \nonumber
&=\sfrac{1}{2}\Big(\big(h+\iota_{\{h \leq \pi_p\}} \big)(x, x^{\loz})+\big(\big(h+\iota_{\{h \leq \pi_p\}}\big)_{p}^{\loz}\circ r\big)(x, x^{\loz})\Big)\\
&\geq \pi_p(x, x^{\loz}).   
\end{align*}
\end{description}
Therefore $\{h\geq\pi_p\}=X \times X^\smalllozenge$. \hfill \qed

\section{$p$-Fitzpatrick transform}\label{p-Fitzpatrick transform}
 The main aim of this section is to prove  some relations between maximal monotone operators and elements of $\Gamma^{\,^\loz}_{p}(X)$.

\begin{definition}{\rm{\rm\cite[Definition 4.1]{MRMA2}}} \label{monpol}{\rm
 We say that $(x,x^\loz)\in X\times X^\smalllozenge$ and $(y,y^\loz) \in X\times X^\smalllozenge$ are \textit{monotonically related}, and we write $(x,x^\loz) \mu (y,y^\loz) $, if $ \langle x^\loz-y^\loz,\overrightarrow{yx}\rangle\geq0 $.  Moreover, $(x,x^\loz) \in X\times X^\smalllozenge$ is \textit{monotonically related} to  $M \subseteq X\times X^\smalllozenge$, denoted by $(x,x^\loz)\mu M$, if
$(x,x^\loz)\mu(y,y^\loz)$  for each $(y,y^\loz) \in M.$

Given  $M\subseteq X \times X^\smalllozenge$, the \textit{monotone polar} of $M$ is defined by \[M^{\mu}:=\{(x,x^\loz)\in X\times X^\smalllozenge : (x,x^\loz)\mu M\}.\]
The
\textit{domain} of a set-valued operator $ T : X \multimap X^\smalllozenge$ is $D(T) :=\{ x\in X : T(x)
\neq\emptyset\}$. The \textit{graph} of $ T $,
denoted by $ \gra (T) $, is $\{(x,x^\loz)\in X\times X^\smalllozenge : x\in D(T), x^\loz \in T(x)\}$.

An operator $T:X\multimap X^\smalllozenge$ is called \textit{monotone} if every $(x,x^\loz),\,(y,y^\loz) \in \gra(T)$ are monotonically related, i.e.   $\gra(T) \subseteq \gra(T)^\mu$. See \cite[Proposition 3.16]{MR0}.
}\end{definition}

\begin{remark} Let $(\mathcal{H},\langg{\cdot}\,|\,{\cdot}\rangg)$ be a Hilbert space, $\mathcal{H}^*$ be its dual space (as dual of a Hadamard space)  and $T:\mathcal{H} \multimap \mathcal{H}^*=\mathcal{H}^\smalllozenge$ be a set-valued operator. Let $(x,x^*) \in  {\gra(T)}$ and $(y,y^*) \in  {\gra(T)}$. Using surjectivity of $j_{_0}$, there exist $u \in \mathcal{H}$ and $v \in \mathcal{H}$ such that $j_{_0} (u)=x^*$ and $j_{_0} (v)=y^*$.
Then\vspace{-0.1cm}
\begin{align*}
\langle x^*-y^*,\overrightarrow{yx}\rangle &=\langle j_{_0} (u)-j_{_0} (v), \overrightarrow{yx} \rangle=\langle \big[\overrightarrow{0u}\big]-\big[\overrightarrow{0v}\big], \overrightarrow{yx} \rangle\\
&=\langle \overrightarrow{0u}-\overrightarrow{0v}, \overrightarrow{yx} \rangle=\langle \overrightarrow{vu}, \overrightarrow{yx} \rangle\\
&=\langg{u-v}\,|\,{x-y}\rangg\\
&=\langg{ j_{_0}^{-1} (x^*)-j_{_0}^{-1} (y^*)\,|\,}{x-y}\rangg.
\end{align*}
Hence $T:\mathcal{H} \multimap \mathcal{H}^*$ is a monotone operator in the sense of Definition \ref{monpol} if and only if $j_{_0}^{-1} \circ T:\mathcal{H} \multimap \mathcal{H}$ is monotone in the classical sense.
\end{remark}

\begin{example}\label{monhyp}
Suppose that $(X, d)$ is a Hadamard space, $a\in X$ and $t_0>0$. Consider $T:X\rightarrow  X^\smalllozenge$ defined by  $T(x)=[t_0\overrightarrow{ax}]$ for each $x\in X$. Then
\[
\langle T(x)-T(y), \overrightarrow{yx}\rangle=\langle [t_0\overrightarrow{ax}]-[t_0\overrightarrow{ay}], \overrightarrow{yx}\rangle=t_0d(x, y)^2\geq0,
\]
i.e., $T$ is monotone.
\end{example}

 \begin{definition}{\rm{\rm\cite[Definition 4.2]{MRMA2}}}{\rm\,A monotone operator $T:X\multimap X^\smalllozenge$ is called \textit{maximal} if there is no monotone operator $S:X\multimap X^\smalllozenge$  that $\gra(S)$ properly contains $\gra(T)$, in other words $\gra(T)=\gra(T)^\mu$, (for details, see \cite[Proposition 3.16]{MR0}).
}
\end{definition}
As it is noted in \cite[Proposition 3.6]{MR0}, using Zorn's lemma, one can deduce that any monotone operator in Hadamard spaces can be extended to a maximal monotone operator. The set of all maximal monotone operators $T:X\multimap X^\smalllozenge$ is denoted by $\mathfrak{M}(X)$.

 \begin{example}  Let $X$ be as in Example \ref{ExMonS}. Consider $T:X\multimap X^\smalllozenge$ with
 \[\gra(T)= \big\{\big(x_n,\big[\overrightarrow{x_{n+1}y_n}\big]\big):n\in\mathbb{N}\big\},\] where $x_n=[(n,1)]$ and $y_n=[(n,0)]$. Then $T$ is monotone but not maximal monotone. See \cite[Example 3.5(i)]{MR0} for details.
\end{example}

 We remind that the Fitzpatrick transform $\mathbf{F}_T:\mathcal{H} \times\mathcal{H}\rightarrow \overline{\mathbb{R}}$ of the set-valued operator $T:\mathcal{H} \multimap \mathcal{H}$ was defined by (\cite{Borwein})
\[\mathbf{F}_T(x,z):= \langg{x\,}|{\,z}\rangg-\inf_{(y,w) \in \gra(T)}\langg{x-y\,}|{\,z-w}\rangg,~~(x,z\in\mathcal{H}).\]

\begin{definition}{\rm{\rm\cite[Definition 5.1]{MRMA2}}}\label{fitzdef} {\rm\, Let $X$ be a Hadamard space and $M \subseteq X \times X^\smalllozenge$. We define  the \textit{$p$-Fitzpatrick transform}  of $M$ as follows:}
\begin{align*}
 \Phi^p_M:&X \times X^\smalllozenge \rightarrow \overline{\mathbb{R}}\\
&(x, x^\loz)\mapsto\sup_{(y,y^\loz)\in M}\big\{\langle x^\loz,\overrightarrow{py}\rangle-\langle y^\loz,\overrightarrow{xy}\rangle\big\}.
 \end{align*}
\end{definition}
Given any operator $T:X \multimap X^\smalllozenge$, the \textit{$p$-Fitzpatrick transform} of $T$ is  $\Phi^p_T:=\Phi^p_{\gra(T)}$.
In other words,
\[\Phi^p_T:X \times X^\smalllozenge \rightarrow \overline{\mathbb{R}},\,(x, x^\loz)\mapsto\sup_{(y,y^\loz)\in \gra(T)}\big\{\langle x^\loz,\overrightarrow{py}\rangle-\langle y^\loz,\overrightarrow{xy}\rangle\big\},~~\big((x, x^\loz)\in X \times X^\smalllozenge\big).\]
\begin{remark}\label{ntsfitz}
\begin{enumerate}\setlength\itemsep{0.05em}
\item[{\rm(i)}]Note that, if\, $\Dom(T)=\emptyset$, then by the convention $\sup \emptyset=-\infty$, we get $\Phi^p_T=-\infty$.
\item[{\rm(ii)}] If $\Dom(T)\neq\emptyset$, then $-\infty\notin \Phi^p_T(X\times X^\smalllozenge)$.
\item[{\rm(iii)}] It follows from {\rm\cite[Proposition 5.2]{MRMA2}} that, for each $(x,x^\loz)\in X \times X^\smalllozenge$, we have:
\begin{equation}\label{FitzInf}
\Phi^p_T(x,x^\loz)=\pi_p(x, x^\loz)-\inf_{(y,y^\loz)\in {\gra(T)}}\langle x^\loz-y^\loz ,\overrightarrow{yx}\rangle.
\end{equation}
\item[{\rm(iv)}] $\Phi^p_T$ is proper and l.s.c. {\rm\cite[Proposition 5.4(i)]{MRMA2}}.
\end{enumerate}
\end{remark}
\begin{remark}\label{AK03}(i) Assume that $X$ is a closed and convex subset of the Hilbert space $(\mathcal{H},\langg{\cdot}\,|\,{\cdot}\rangg)$ with $a \in \Int(X)$ and let $T:X \multimap X^\smalllozenge$ be a set-valued operator. For each $(x,x^\loz) \in X \times X^\smalllozenge$,  using surjectivity of $j_a$, there exists $u \in \mathcal{H}$ such that $j_a (u)=x^\loz$.
Likewise, for each $(y,y^\loz)\in {\gra(T)}$, there exists $v \in \mathcal{H}$ such that $j_a (v)=y^\loz$.
Then
\begin{align*}
\Phi^p_T(x,x^\loz)&=\Phi^p_T(x,j_a (u))\\&=\pi_p(x, x^\loz)-\inf_{(y,y^\loz)\in {\gra(T)}}\langle x^\loz-y^\loz ,\overrightarrow{yx}\rangle\\
&=\langle j_a (u), \overrightarrow{px} \rangle-\inf_{j_a (v)\in {T(y)}}\langle j_a (u)-j_a (v), \overrightarrow{yx} \rangle\\
&=\langg{u\,}|{\,x-p}\rangg -\inf_{j_a (v)\in {T(y)}} \langg{u-v\,}|{\,x-y}\rangg\\
&=\big(\langg{u\,}|{\,x}\rangg -\inf_{v \in (j_a ^{-1} \circ T)(y)}\langg{u-v\,}|{\,x-y}\rangg\big)-\langg{u\,}|{\,p}\rangg\\
&=\mathbf{F}_{j_a ^{-1} \circ T}(x,u)-\langg{u\,}|{\,p}\rangg\\
&=\mathbf{F}_{j_a ^{-1} \circ T}\big(x,j_a ^{-1}(x^\loz)\big)-\langg{j_a ^{-1}(x^\loz)\,}|{\,p}\rangg.
\end{align*}

(ii) As a special case of (i), when $X=\mathcal{H}$ and $p=0$,  we have:
\begin{align*}\Phi^0_T(x,x^\loz)&=\Phi^0_T(x,j_0 (u))=\Phi^0_T\big(x,[\overrightarrow{0u}]\big)=\mathbf{F}_{j_0 ^{-1} \circ T}(x,u)-\langg{u\,}|{\,0}\rangg
=\mathbf{F}_{j_0 ^{-1} \circ T}\big(x,j_0 ^{-1}(x^\loz)\big).
\end{align*}
Similar to Remark \ref{AK01}(ii), using Remark \ref{AK01}(ii), we can write $\Phi^0_T=\mathbf{F}_{T}$.
\end{remark}
\begin{example} Let $X$ be as in Example \ref{ExMonS}.
For each $n\in \mathbb{N}$, set $x_n:=\big[\big(n,\frac{1}{n}\big)\big]$.
Let $T:X \multimap X^\smalllozenge$ be the operator with
\[\gra(T):=\big\{(x_n,[\overrightarrow{x_nx_{n+1}}]):n\in \mathbb{N}\big\}\subseteq X\times X^\smalllozenge.\]
Let $p:=[(n_0,t_0)], x:=[(1,0)] \in X$, $x^\loz:=\Big[\overrightarrow{\big(2,\frac{2}{3}\big)(3,1)}\Big] \in X^\smalllozenge$ and $(y, y^\loz)\in \gra(T)$. Then there is $n\in\mathbb{N}$ such that $y=(n,\frac{1}{n})$ and \[y^\loz=\Big[\overrightarrow{\Big(n,\sfrac{1}{n}\Big)\Big(n+1,\sfrac{1}{n+1}\Big)}\Big].\] Therefore,
\begin{align*}
\langle x^\loz-y^\loz ,\overrightarrow{yx}\rangle &=\langle x^\loz,\overrightarrow{yx}\rangle-\langle y^\loz ,\overrightarrow{yx}\rangle\\
&=\Big\langle \overrightarrow{\big(2,\ssfrac{2}{3}\big)(3,1)},\overrightarrow{\big(n,\ssfrac{1}{n}\big)(1,0)}\Big\rangle-\Big\langle \overrightarrow{\big(n,\ssfrac{1}{n}\big)\big(n+1,\ssfrac{1}{n+1}\big)} ,\overrightarrow{\big(n,\ssfrac{1}{n}\big)(1,0)}\Big\rangle\\
&=\begin{cases}
-\frac{5}{9} & n=3,
\\[1mm]
\frac{1}{3n}    &n\notin \{2,3\},
\\[1mm]
\frac{5}{6} & n=2,
\end{cases}
~~-
\begin{cases}
\frac{3}{2} & n=1,
\\[1mm]
\frac{2n+1}{n^2(n+1)}   &n \neq 1,
\end{cases}=
\begin{cases}
-\frac{7}{6} & n=1,
\\
\frac{5}{12}   &n=2,
\\
-\frac{3}{4} & n=3,
\\
\frac{n^2-5n-3}{3n^2(n+1)}   &o.w.,
\end{cases}
\end{align*}
and
\begin{align*}
\inf_{(y,y^\loz)\in {\gra(T)}}\langle x^\loz-y^\loz ,\overrightarrow{yx}\rangle&=\min\Big\{-\sfrac{7}{6},\sfrac{5}{12} ,-\sfrac{3}{4},\inf_{n \geq 4}\sfrac{n^2-5n-3}{3n^2(n+1)}\Big\}=-\sfrac{7}{6}.
\end{align*}
Moreover,
\begin{align*}
\pi_p(x, x^\loz)&=\langle x^\loz,\overrightarrow{px}\rangle=\Big\langle  \overrightarrow{\big(2,\ssfrac{2}{3}\big)(3,1)}, \overrightarrow{(n_0,t_0)(1,0)}\Big\rangle=\begin{cases}
\frac{5}{3}t_{0} & n_0=2,
\\[1mm]
-\frac{5}{3}t_0 & n_0=3,
\\[1mm]
\frac{1}{3}t_0    &n_0\notin \{2,3\}.
\end{cases}
\end{align*}
Hence,
\begin{align*}
\Phi^p_T(x,x^\loz)&=\pi_p(x, x^\loz)-\inf_{(y,y^\loz)\in {\gra(T)}}\langle x^\loz-y^\loz ,\overrightarrow{yx}\rangle=\begin{cases}
\frac{5}{3}t_0+\frac{7}{6} & n_0=2,
\\[1mm]
-\frac{5}{3}t_0+\frac{7}{6} & n_0=3,
\\[1mm]
\frac{1}{3}t_0+\frac{7}{6}    &n_0\notin \{2,3\}.
\end{cases}
\end{align*}
\end{example}
\begin{example} Consider the operator $T: \mathbb{H}^2\multimap \big(\mathbb{H}^2\big)^\smalllozenge $ with the graph
\[\gra(T):=\Big\{\big(\alpha(t),\big[\overrightarrow{\alpha(t)\alpha(t+1)}\big]\big):\,t \geq 0\Big\},\]
where $\alpha(t):=(\sinh t, 0, \cosh t)$. Take $p=(1,-1,\sqrt{3})$ and  $x=(0,0,1)$ in the real hyperbolic $2$-space $\mathbb{H}^2$ and $x^\loz=\Big[\overrightarrow{\big(1,0,\sqrt{2}\big)\big(0,-1,\sqrt{2}\big)}\Big] \in \big(\mathbb{H}^2\big)^\smalllozenge$. Then
\begin{align*}
\langle x^\loz, \overrightarrow{px} \rangle
&=\Big\langle \Big[\overrightarrow{\big(1,0,\sqrt{2}\big)\big(0,-1,\sqrt{2}\big)}\Big], \overrightarrow{\big(1,-1,\sqrt{3}\big)(0,0,1)}\Big \rangle\\
&=\Big\langle \overrightarrow{\big(1,0,\sqrt{2}\big)\big(0,-1,\sqrt{2}\big)}, \overrightarrow{\big(1,-1,\sqrt{3}\big)(0,0,1)}\Big \rangle\\
&\!=\!\frac{1}{2}\!\Big(\!\!\big(\!\cosh^{-1}(\!\sqrt{2})\!\big)^2\!\!+\!\!\big(\!\cosh^{-1}(\!\sqrt{6}\!\!-\!\!1)\!\big)^2\!-\!\big(\!\cosh^{-1}(\!\sqrt{6}\!-\!1)\!\big)^2\!\!-\!\!\big(\!\cosh^{-1}(\!\sqrt{2})\!\big)^2\!\Big)\\
&=0.
\end{align*}
For each $(y,y^\loz)=\big(\alpha(t),\big[\overrightarrow{\alpha(t)\alpha(t+1)}\big]\big) \in \gra(T)$,
\begin{align*}
\langle x^\loz,\overrightarrow{yx}\rangle &=\Big\langle \Big[\overrightarrow{\big(1,0,\sqrt{2}\big)\big(0,-1,\sqrt{2}\big)}\Big], \overrightarrow{(\sinh t, 0, \cosh t)(0,0,1)}\Big \rangle\\
&=\Big\langle \overrightarrow{\big(1,0,\sqrt{2}\big)\big(0,-1,\sqrt{2}\big)}, \overrightarrow{(\sinh t, 0, \cosh t)(0,0,1)}\Big \rangle\\
&=\sfrac{1}{2}\Big(\big(\cosh^{-1}(\sqrt{2})\big)^2+\big(\cosh^{-1}(\sqrt{2}\cosh t)\big)^2\\
&~~~~~~~~~-\big(\cosh^{-1}(\sqrt{2}\cosh t-\sinh t)\big)^2-\big(\cosh^{-1}(\sqrt{2})\big)^2\Big)\\
&=\sfrac{1}{2}\Big(\big(\cosh^{-1}(\sqrt{2}\cosh t)\big)^2-\big(\cosh^{-1}(\sqrt{2}\cosh t-\sinh t)\big)^2\Big).
\end{align*}
On the other hand,
\begin{align*}
\langle y^\loz,\overrightarrow{yx}\rangle &=\Big\langle \Big[\overrightarrow{(\sinh t, 0, \cosh t)(\sinh (t+1), 0, \cosh (t+1))}\Big], \overrightarrow{(\sinh t, 0, \cosh t)(0,0,1)}\Big \rangle\\[1mm]
&=\Big\langle \overrightarrow{(\sinh t, 0, \cosh t)(\sinh (t+1), 0, \cosh (t+1))}, \overrightarrow{(\sinh t, 0, \cosh t)(0,0,1)}\Big \rangle\\[1mm]
&=\sfrac{1}{2}\Big(\big(\cosh^{-1}(\cosh t)\big)^2+\big(\cosh^{-1}\big(\cosh (t+1)\cosh t-\sinh (t+1)\sinh t\big)\big)^2\\[1mm]
&~~~~~-\big(\cosh^{-1}\big(\cosh (t+1)\big)\big)^2\Big)\\[1mm]
&=\sfrac{1}{2}\big(t^2+1-(t+1)^2\big)\\[1mm]
&=-t.
\end{align*}
Therefore,
\begin{align*}
\Phi^p_T(x,x^\loz)&=\pi_p(x, x^\loz)-\inf_{(y,y^\loz)\in {\gra(T)}}\langle x^\loz-y^\loz ,\overrightarrow{yx}\rangle\\
&=-\inf_{t \geq 0} \Big\{\sfrac{1}{2}\Big(\big(\cosh^{-1}(\sqrt{2}\cosh t)\big)^2-\big(\cosh^{-1}(\sqrt{2}\cosh t-\sinh t)\big)^2\Big)+t\Big\}\\
&=0.
\end{align*}
Thus $\Phi^p_T\equiv0$.
\end{example}
\begin{proposition}\label{fitzflconvex} %{\rm({Compare with \cite[]{}})}
Let  $T:X \multimap X^\smalllozenge$ be an operator.
\begin{enumerate}\setlength\itemsep{0.15em}
\item[{\rm(i)}] Let $a,b\in X$ be such that $\{a, b\}\times\Ran(T)$ has the $\mathcal{F}_l$-property. Then \[\Phi^p_T(c(\lambda),c^\loz(\lambda))\leq(1-\lambda)\Phi^p_T(a,a^\loz)+\lambda\Phi^p_T(b,b^\loz),\]
where $a^\loz, b^\loz\in X^\smalllozenge$, $c(\lambda):=(1-\lambda)a\oplus \lambda b$ and $c^\loz(\lambda):=(1-\lambda)a^\loz+ \lambda b^\loz$.
\item[{\rm(ii)}] Let $C \subseteq X \times X^\smalllozenge$ be a convex set with the $\mathcal{F}_l$-property and $\Ran(T) \subseteq \Ran(C)$. Then $\Phi^p_T+\iota_C$ is convex.
\item[{\rm(iii)}] $\Phi^p_T$ is convex if $X \times \Ran(T)$ has the $\mathcal{F}_l$-property.
\item[{\rm(iv)}] $\Phi^p_T$ is convex if $X$ is a flat Hadamard space.
\end{enumerate}
\end{proposition}
\begin{description}\setlength\itemsep{0.25em}\setlength{\itemindent}{-1.5em}
\item[{\textit{\textbf{Proof.}}} {\rm(i)}] By the definition of $\Phi^p_T$ and the $\mathcal{F}_l$-property we have:
\begin{align*}
\Phi^p_T&(c(\lambda),c^\loz(\lambda))=\sup_{(y,y^\loz)\in \gra(T)}\big\{\langle c^\loz(\lambda),\overrightarrow{py}\rangle-\langle y^\loz,\overrightarrow{c(\lambda)y}\rangle\big\}\\
&=\sup_{(y,y^\loz)\in \gra(T)}\big\{\langle c^\loz(\lambda),\overrightarrow{py}\rangle+\langle y^\loz,\overrightarrow{yc(\lambda)}\rangle\big\}\\
&\leq\sup_{(y,y^\loz)\in \gra(T)}\big\{(1-\lambda)\langle a^\loz,\overrightarrow{py}\rangle+\lambda \langle b^\loz,\overrightarrow{py}\rangle+(1-\lambda)\langle y^\loz,\overrightarrow{ya}\rangle+\lambda \langle y^\loz,\overrightarrow{yb}\rangle\big\}\\
&\leq(1-\lambda)\!\!\sup_{(y,y^\loz)\in \gra(T)}\!\!\big\{\langle a^\loz,\overrightarrow{py}\rangle-\langle y^\loz,\overrightarrow{ay}\rangle\big\}\!+\!\lambda\!\!\sup_{(y,y^\loz)\in \gra(T)}\!\!\big\{\langle b^\loz,\overrightarrow{py}\rangle\!-\!\langle y^\loz,\overrightarrow{by}\rangle\big\}\\
&=(1-\lambda)\Phi^p_T(a,a^\loz)+\lambda\Phi^p_T(b,b^\loz),
\end{align*}
and hence (i) is proved.

\item[{\rm (ii)}]  Suppose that $(a,a^\loz),(b,b^\loz) \in C$  are arbitrary and fixed. For each $\lambda \in [0,1]$, set
$c(\lambda)\!:=\!(1-\lambda)a\oplus \lambda b$ and $c^\loz(\lambda)\!:=\!(1-\lambda)a^\loz\!+\! \lambda b^\loz$. In view of (i), we get:
\begin{align*}
(\Phi^p_T+\iota_C)(c(\lambda),c^\loz(\lambda))&=\Phi^p_T(c(\lambda),c^\loz(\lambda))\\
&\leq(1-\lambda)\Phi^p_T(a,a^\loz)+\lambda\Phi^p_T(b,b^\loz)\\
&=(1-\lambda)(\Phi^p_T+\iota_C)(a,a^\loz)+\lambda(\Phi^p_T+\iota_C)(b,b^\loz),
\end{align*}
which it implies that $\Phi^p_T+\iota_C$ is convex.

\item[{\rm(iii)}]  The proof is similar to that of (ii).

\item[{\rm(iv)}] It is an immediate consequence of (iii) and Theorem \ref{frtg}. Also, it can be deduced from (ii).
\end{description}

The following result is a direct consequence of \cite[Proposition 5.3]{MRMA2}.
\begin{proposition}\label{MainFpM}
 Let $T:X\multimap X^\smalllozenge$. Then
 \begin{enumerate}\setlength\itemsep{0.15em}
 \item[{\rm(i)}] $\gra(T)^{\mu}=\{ \Phi^p_T\leq\pi_p\}$.
\item[{\rm(ii)}] If $T$ is monotone, then $\gra(T)\subseteq \{ \Phi^p_T=\pi_p\}$.
\item[{\rm(iii)}] If $T$ is maximal monotone, then  $X\!\times\! X^\smalllozenge\!\setminus \gra(T)\!\subseteq\!\{ \Phi^p_T\!>\!\pi_p\}$; in other words, $\{ \Phi^p_T\!\leq\!\pi_p\}\!\subseteq\!\gra(T)$.
\item[{\rm(iv)}] If $T$ is maximal monotone, then $\{ \Phi^p_T\geq\pi_p\}=X\times X^\smalllozenge$.
\item[{\rm(v)}] If $T$ is maximal monotone, then   $\{ \Phi^p_T=\pi_p\}=\{\Phi^p_T\leq\pi_p\}=\gra(T)$.
\item[{\rm(vi)}] If $\{ \Phi^p_T=\pi_p\}=\gra(T)$ and $\{ \Phi^p_T\geq\pi_p\}=X\times X^\smalllozenge$, then $T$ is maximal monotone.
\end{enumerate}
\end{proposition}
\begin{corollary}\label{maxFitz-Iden}
Let $T, S\in\mathfrak{M}(X)$ and $ \Phi^p_T= \Phi^p_S$. Then $T=S$.
\end{corollary}
\begin{proof} Let $T$ and $S$ be maximal monotone operators with $\Phi^p_T= \Phi^p_S$. It follows from Proposition \ref{MainFpM}(v) that
\begin{equation*}
\gra(T)=\{\Phi^p_T=\pi_p\}=\{ \Phi^p_S=\pi_p\}=\gra(S).
\end{equation*}
\end{proof}

\begin{proposition}\label{phitx111} Let  $T\!:\!X\!\multimap\!X^\smalllozenge$ be an operator with $ \gra(T)\!\neq\!\emptyset$. The following assertions hold.
\begin{enumerate}\setlength\itemsep{0.15em}
%\item[(i)]  $\Phi^p_T$ is lower semi-continuous and proper.
\item[{\rm(i)}] $\Phi^p_T=(\iota_{\gra(T)} +\pi_p)_p^{\loz}\circ r $.
\item[{\rm(ii)}] $(\Phi^p_T)_p^{\loz} \circ r \leq \pi_p+ \iota_{\gra(T)}$.
\item[{\rm(iii)}] If $T$ is monotone, then $\Phi^p_T \leq (\Phi^p_T)_p^{\loz} \circ r \leq \pi_p+ \iota_{\gra(T)}$, with equality in  $ \gra(T)$.
\item[{\rm(iv)}] If $T$ is monotone and $\{\Phi^p_T\!\leq\!\pi_p\}\!=\!\gra(T)$,
then \[\Phi^p_T\!=\!\big(\Phi^p_T+\iota_{\{\Phi^p_T\leq \pi_p\}}\big)_p^{\loz}\circ r.\]
\item[{\rm(v)}] Suppose that  $T\in\mathfrak{M}(X)$ and $C:=\{\Phi^p_T= \pi_p\}=\{\Phi^p_T\leq \pi_p\}=\gra(T)$. Then $\Phi^p_T=\big(\Phi^p_T+\iota_{C}\big)_p^{\loz}\circ r$.
\end{enumerate}
\end{proposition}
\begin{description}\setlength\itemsep{0.1em}\setlength{\itemindent}{-2.5em}
\item[{\textit{\textbf{Proof.}} {\rm(i)}}]  Let $(x,x^\loz) \in X\times X^\smalllozenge$. Then
\begin{align*}
\Phi^p_T(x,x^{\loz})&=\sup_{(y,y^{\loz})\in \gra (T)}\big\{\langle y^{\loz} ,\overrightarrow{px}\rangle+\langle x^{\loz} ,\overrightarrow{py}\rangle-\langle y^{\loz} ,\overrightarrow{py}\rangle\big\}\\
&=\sup_{(y,y^{\loz})\in X \times X^\smalllozenge}\big\{\langle y^{\loz} ,\overrightarrow{px}\rangle+\langle x^{\loz} ,\overrightarrow{py}\rangle- (\iota_{\gra(T)} +\pi_p)(y,y^{\loz})\big\}\\
&=(\iota_{\gra(T)} +\pi_p)_p^{\loz}(x^{\loz},x)\\
&=\big((\iota_{\gra(T)} +\pi_p)_p^{\loz}\circ r\big)(x,x^{\loz}).
\end{align*}

\item[{\rm(ii)}]
Applying Proposition \ref{FenchelYoungIneqth}(i) to $h:=\iota_{\gra(T)} +\pi_p$, using (i), we get:
\[
(\iota_{\gra(T)} +\pi_p)(x,x^{\loz})+\Phi^p_T(y,y^{\loz})\geq \langle y^{\loz} ,\overrightarrow{px}\rangle+\langle x^{\loz} ,\overrightarrow{py}\rangle, \forall (y,y^{\loz})\in X\times X^\smalllozenge;
\]
i.e.,
\[
(\iota_{\gra(T)} +\pi_p)(x,x^{\loz})\geq \langle y^{\loz} ,\overrightarrow{px}\rangle+\langle x^{\loz} ,\overrightarrow{py}\rangle-\Phi^p_T(y,y^{\loz}), \forall (y,y^{\loz})\in X\times X^\smalllozenge.
\]
Therefore,
\begin{align*}
\big((\Phi^p_T)_p^{\loz} \circ r \big)(x,x^{\loz})&=\sup_{(y,y^{\loz})\in X \times X^\smalllozenge}\big\{\langle y^{\loz} ,\overrightarrow{px}\rangle+\langle x^{\loz} ,\overrightarrow{py}\rangle-\Phi^p_T(y,y^{\loz} )\big\}\\
&\leq(\iota_{\gra(T)} +\pi_p)(x,x^{\loz}).
\end{align*}
\item[{\rm(iii)}]  Let $(x,x^{\loz}) \in X\times X^\smalllozenge$. From Proposition \ref{MainFpM}(ii) and Definition \ref{$p$-Fenchel} we see that
\begin{align*}
\Phi^p_T(x,x^{\loz})&=\sup_{(y,y^{\loz})\in \gra (T)}\big\{\langle y^{\loz} ,\overrightarrow{px}\rangle+\langle x^{\loz} ,\overrightarrow{py}\rangle-\langle y^{\loz} ,\overrightarrow{py}\rangle\big\}\\
&= \sup_{(y,y^{\loz})\in \gra (T)}\big\{\langle y^{\loz} ,\overrightarrow{px}\rangle+\langle x^{\loz} ,\overrightarrow{py}\rangle-\Phi^p_T(y,y^{\loz})\big\}\\
&\leq \sup_{(y,y^{\loz})\in X \times X^\smalllozenge}\big\{\langle y^{\loz} ,\overrightarrow{px}\rangle+\langle x^{\loz} ,\overrightarrow{py}\rangle-\Phi^p_T(y,y^{\loz} )\big\}\\
&=(\Phi^p_T)_p^{\loz}(x^{\loz},x)\\
&=\big((\Phi^p_T)_p^{\loz} \circ r \big)(x,x^{\loz}).
\end{align*}
Thus $\Phi^p_T \leq (\Phi^p_T)_p^{\loz} \circ r$ on $X\times X^\smalllozenge$.
So, (ii) completes the proof of the first statement.

The second assertion follows from Proposition \ref{MainFpM}(ii).

\item[{\rm(iv)}] Let $T:X\multimap X^\smalllozenge$ be monotone and $\gra(T)=\{\Phi^p_T\leq \pi_p\}$. By Proposition \ref{MainFpM}(ii), we have $\gra(T)=\{\Phi^p_T=\pi_p\}$. So, we can write:
\begin{align*}
\big(\big(\Phi^p_T+\iota_{\{\Phi^p_T\leq \pi_p\}}\big)_p^{\loz}&\circ r\big)(x,x^{\loz})=\big(\Phi^p_T+\iota_{\{\Phi^p_T\leq \pi_p\}}\big)_p^{\loz}(x^{\loz},x)\\
&=\sup_{(y,y^{\loz})\in X \times X^\smalllozenge}\big\{\langle y^{\loz} ,\overrightarrow{px}\rangle+\langle x^{\loz} ,\overrightarrow{py}\rangle-\big(\Phi^p_T+\iota_{\gra(T)}\big)(y,y^{\loz})\big\}\\
&=\sup_{(y,y^{\loz})\in\gra(T)}\big\{\langle y^{\loz} ,\overrightarrow{px}\rangle+\langle x^{\loz} ,\overrightarrow{py}\rangle-\Phi^p_T(y,y^{\loz})\big\}\\
&=\sup_{(y,y^{\loz})\in \gra(T)}\big\{\langle y^{\loz} ,\overrightarrow{px}\rangle+\langle x^{\loz} ,\overrightarrow{py}\rangle-\langle y^\loz ,\overrightarrow{py}\rangle\big\}\\
&=\Phi^p_T(x,x^{\loz}).
\end{align*}
\item[{\rm(v)}] It is an immediate consequence of (iii) and Proposition \ref{MainFpM}(v).%\hfill \qed
\end{description}

\begin{proposition}\label{FTPhi}
Suppose that  $X\times\Ran(T)$ has the $\mathcal{F}_l$-property, where $T\in\mathfrak{M}(X)$. Then $\Phi^p_T \in \Gamma^{\,^\loz}_{p}(X)$. In particular, if $\,X$ is a flat  Hadamard space, then $\Phi^p_T \in \Gamma^{\,^\loz}_{p}(X)$.
\end{proposition}
\begin{proof} Using Remark \ref{ntsfitz}(iv) and Proposition \ref{fitzflconvex}(iii), we obtain $\Phi^p_T \in \Gamma(X \times X^\loz)$.
Furthermore, it follows from Proposition \ref{phitx111}(v) that \[\Phi^p_T(x,x^{\loz})=\big((\Phi^p_T+\iota_{\{\Phi^p_T\leq \pi_p\}})_p^{\loz}\circ r\big)(x,x^{\loz}),\] for each $(x,x^{\loz}) \in X \times X^\smalllozenge$.
Thus $\Phi^p_T \in \Gamma^{\,^\loz}_{p}(X)$. Finally, the second assertion is an easy consequence of the first assertion and Proposition \ref{fitzflconvex}(iv).
\end{proof}

\begin{proposition}\label{Prop-Shp}  Associate to any mapping $h: X \times X^\smalllozenge \rightarrow \overline{\mathbb{R}}$, the set-valued map $S_{h,p}\!:\!X \multimap X^\smalllozenge$ defined by
\begin{equation}\label{Sh1}
S_{h,p} (x):=\{x^\loz \in X^\smalllozenge: h(x,x^{\loz})=\pi_p(x,x^{\loz})\}.
\end{equation}
If $h\in \Gamma^{\,^\loz}_{p}(X)$, then $\Phi^p_{S_{h,p}}=h$.\end{proposition}
\begin{proof}
From \eqref{Sh1}, Propositions \ref{FenchelYoungIneqth}(iii) and \eqref{Gamma-Loz-p} we obtain:
\begin{align*}
\Phi^p_{S_{h,p}}(x,x^{\loz})&=\sup_{(y,y^{\loz})\in \gra (S_{h,p})}\big\{\langle y^{\loz} ,\overrightarrow{px}\rangle+\langle x^{\loz} ,\overrightarrow{py}\rangle-\langle y^{\loz} ,\overrightarrow{py}\rangle\big\}\\
&=\sup_{(y,y^{\loz})\in \{h=\pi_p\}}\big\{\langle y^{\loz} ,\overrightarrow{px}\rangle+\langle x^{\loz} ,\overrightarrow{py}\rangle-h(y,y^{\loz})\big\}\\
&=\sup_{(y,y^{\loz})\in \{h\leq \pi_p\}}\big\{\langle y^{\loz} ,\overrightarrow{px}\rangle+\langle x^{\loz} ,\overrightarrow{py}\rangle-h(y,y^{\loz})\big\}\\
&=\sup_{(y, y^{\loz})\in X \times X^\smalllozenge}\big\{\langle y^{\loz} ,\overrightarrow{px}\rangle+\langle x^{\loz} ,\overrightarrow{py}\rangle-\big(h+\iota_{\{h \leq \pi_p\}}\big)(y,y^{\loz})\big\}\\
&=\big(h+\iota_{\{h \leq \pi_p\}}\big)_p^{\loz}(x^{\loz},x)\\
&=\big(\big(h+\iota_{\{h \leq \pi_p\}}\big)_p^{\loz}\circ r\big)(x,x^{\loz})\\
&=h(x,x^{\loz});
\end{align*}
i.e., $\Phi^p_{S_{h,p}}=h$.
\end{proof}
\begin{theorem}\label{Prop-Shp-02}
Let $h\in \Gamma^{\,^\loz}_{p}(X)$. Then $S_{h,p}\in\mathfrak{M}(X)$.
\end{theorem}
\begin{proof}
 Let $h\in \Gamma^{\,^\loz}_{p}(X)$ and let $(x,x^{\loz}), (y,y^{\loz}) \in \gra(S_{h,p} )$. So, $(x,x^{\loz}), (y,y^{\loz})\in \{h = \pi_p\}$ and so by using  \eqref{Sh1} and Proposition \ref{FenchelYoungIneqth}(i) we get:
\begin{align*}
\langle x^{\loz}-y^{\loz} ,\overrightarrow{yx}\rangle &=\langle x^{\loz}-y^{\loz} ,\overrightarrow{px}-\overrightarrow{py}\rangle \\
&=\langle x^{\loz} ,\overrightarrow{px}\rangle+\langle y^{\loz} ,\overrightarrow{py}\rangle-\langle x^{\loz} ,\overrightarrow{py}\rangle-\langle y^{\loz} ,\overrightarrow{px}\rangle\\
&=h(x,x^{\loz})+h(y,y^{\loz})-\langle x^{\loz} ,\overrightarrow{py}\rangle-\langle y^{\loz} ,\overrightarrow{px}\rangle\\
&=\big(h+\iota_{\{h \leq \pi_p\}}\big)(x,x^{\loz})\!+\!\big(h+\iota_{\{h \leq \pi_p\}}\big)_p^{\loz}(y^{\loz},y)\!-\!\langle x^{\loz} ,\overrightarrow{py}\rangle\!-\!\langle y^{\loz} ,\overrightarrow{px}\rangle\\
&\geq 0.
\end{align*}
Hence $S_{h,p}$ is monotone.

Now, let $(x,x^{\loz}) \in \big(\gra(S_{h,p})\big)^\mu$.
 Propositions \ref{MainFpM}(i),  \ref{Prop-Shp} and  \ref{FenchelYoungIneqth}(iii) imply that
\begin{align*}
\pi_p(x, x^{\loz})\geq \Phi^p_{S_{h,p}}(x,x^{\loz})=h(x,x^{\loz})\geq \pi_p(x, x^{\loz}).  &
\end{align*}
Thus $h(x,x^{\loz})=\pi_p(x, x^{\loz})$, which means that $(x,x^{\loz})\in \gra(S_{h,p})$ and so  $S_{h,p}$ is maximal.
\end{proof}

\begin{theorem}\label{Prop-Shp-03}
Let $T\in\mathfrak{M}(X)$ be such that $X\times\Ran(T)$ has the $\mathcal{F}_l$-property. Then there exists a unique $h \in \Gamma^{\,^\loz}_{p}(X) $ for which  $ T=S_{h,p}$.
\end{theorem}
\begin{proof}
Set $h:=\Phi^p_T$. It follows from  Proposition \ref{FTPhi} that $h \in \Gamma^{\,^\loz}_{p}(X) $. By Proposition \ref{MainFpM}(v) we have:
\begin{align*}
\gra(S_{h, p})&=\gra(S_{\Phi^p_T,p})\\
&=\{(x, x^\loz)\in X\times X^\smalllozenge:\Phi^p_T(x,x^\loz)=\pi_p(x,x^\loz) \}\\
&=\gra(T).
\end{align*}
 Finally, for the uniqueness, let $h_1,h_2 \in \Gamma^{\,^\loz}_{p}(X)$ be such that $T={S_{h_1,p}}={S_{h_2,p}}$. Using Proposition \ref{Prop-Shp}, we get $h_1=\Phi^p_{S_{h_1,p}}=\Phi^p_{S_{h_2,p}}=h_2$.
\end{proof}

\begin{corollary}\label{Cor01} Let $X$ be a flat Hadamard space. Then $T:X \multimap X^\smalllozenge$ is maximal monotone if and only if there exists a unique $h \in \Gamma^{\,^\loz}_{p}(X) $ such that  $ T=S_{h,p}$.
\end{corollary}
 \begin{proof}
 The assertion follows from Theorems \ref{Prop-Shp-02} and  \ref{Prop-Shp-03}.
 \end{proof}

\begin{corollary}\label{Cor02}
In a flat Hadamard space $X$, there is a one-to-one correspondence between $\mathfrak{M}(X)$ and $ \Gamma^{\,^\loz}_{p}(X)$.
\end{corollary}
\begin{proof}
Consider the mapping $\Psi:\mathfrak{M}(X)\rightarrow \Gamma^{\,^\loz}_{p}(X)$, defined by $\Psi(T)=\Phi^p_T$.
By Corollary \ref{maxFitz-Iden}, $\Psi$ is one-to-one. Now, consider $h\in \Gamma^{\,^\loz}_{p}(X)$. From Theorem \ref{Prop-Shp-02}, we have
$S_{h,p}\in\mathfrak{M}(X)$. Finally, Proposition \ref{Prop-Shp} implies that
$\Psi(S_{h,p})=\Phi^p_{S_{h,p}}=h$.
\end{proof}
\begin{remark}
Note that, using Remark \ref{Rem-Flat} about flat Hadamard spaces, one can prove that previous corollaries (\ref{Cor01} $\&$  \ref{Cor02}) are equivalent to the classical results about Fitzpatrick transforms in Hilbert spaces.
\end{remark}

\section{Conclusion}

In this paper, we have investigated Fenchel conjugate transform and Fenchel-Young inequality for extended real-valued functions on $X\times X^{\smalllozenge}$, where $X$ is a Hadamard space and $X^{\smalllozenge}$ is its linear dual space. For an arbitrary and fixed  $p\in X$, the $p$-Fitzpatrick transform for monotone set-valued operators from a Hadamard space to its linear dual space is introduced and their basic properties are examined. In addition, it is shown that the $p$-Fitzpatrick transform on closed and convex subsets of a Hilbert space with non-empty interior can be expressed in terms of the classical Fitzpatrick transform and in the special case $p=0$, they coincide. Moreover, for each $x\in X$, we have associated the set $S_{h,p} (x):=\{h(x,\cdot)=\pi_p(x,\cdot)\}$  to any extended real-valued function $h$ on $X \times X^\smalllozenge$. It is proved that $S_{h,p}$ is a  maximal monotone set-valued mapping from $X$ to $X^\smalllozenge$, which its $p$-Fitzpatrick transform represents $h$ provided that  
$h$ is a proper, l.s.c. and convex map which equals to $\big(h+\iota_{\{h \leq \pi_p\}}\big)_{p}^{\loz}\!\circ\! r \text{~on~} X\times X^\smalllozenge\!$. On the other hand, for each $T\in\mathfrak{M}(X)$  there exists a unique $h \in \Gamma^{\,^\loz}_{p}(X) $ for which  $ T=S_{h,p}$, whenever $X\times\Ran(T)$ has the $\mathcal{F}_l$-property.

For the future works, one can define resolvent of monotone operators and apply them to finding the common zero
of monotone operators and also common fixed point of nonexpansive mappings with applications in convex optimization.
Some of other ideas for the future works in this field are:
\begin{itemize}
\item Convex subdifferential theory in Hadamard spaces: including
maximality of the subdifferential (specially in non-flat Hadamard spaces), new
generalizations and applications of subdifferential sum rule.
\item Duality mappings and surjectivity results: generalizations of applications.
\item Cyclic and $n$-cyclic monotonicity.
\item Maximal monotonicity of bifunctions.
\item Fitzpatrick functions of bifunctions and related representations results.

\end{itemize}

\end{document}